\documentclass[11pt,a4paper]{amsart}
\usepackage{amssymb}
\usepackage[all]{xy}
\usepackage{graphicx}
\usepackage[colorlinks=false,hyperindex]{hyperref} 

\newtheorem{theorem}{Theorem}[section]
\newtheorem{lem}[theorem]{Lemma}
\newtheorem{thm}[theorem]{Theorem}

\newtheorem{cor}[theorem]{Corollary}

\newtheorem{ex}[theorem]{Example}
\newtheorem{defn}[theorem]{Definition}
\newtheorem{rem}[theorem]{Remark}

\newcommand{\Hom}{{\rm Hom}}
\newcommand{\End}{{\rm End}}
\newcommand{\Gen}{{\rm Gen}}

\begin{document}

\title[Correspondences of coclosed submodules]{Correspondences of coclosed submodules}

\author[S. Crivei]{Septimiu Crivei}

\address{Faculty of Mathematics and Computer Science,``Babe\c s-Bolyai" University, Str. Mihail Kog\u alni-ceanu 1,
400084 Cluj-Napoca, Romania} \email{crivei@math.ubbcluj.ro}

\author[H. Inank\i l]{Hatice Inank\i l}

\address{Department of Mathematics,  Gebze Institute of Technology, \c{C}ayirova  Campus, 41400 Gebze-Kocaeli, 
Turkey} \email{hinankil@gyte.edu.tr}

\author[M.T. Ko\c san]{M. Tamer Ko\c san}

\address{Department of Mathematics,  Gebze Institute of Technology, \c{C}ayirova  Campus, 41400 Gebze-Kocaeli, 
Turkey} \email{mtkosan@gyte.edu.tr}

\author[G. Olteanu]{Gabriela Olteanu}

\address{Department of Statistics-Forecasts-Mathematics, ``Babe\c s-Bolyai'' University,
Str. T. Mihali 58-60, 400591 Cluj-Napoca, Romania}
\email{gabriela.olteanu@econ.ubbcluj.ro}


\begin{abstract} We establish an order-preserving bijective correspondence between the sets of coclosed elements of some
bounded lattices related by suitable Galois connections. As an application, we deduce that if $M$ is a finitely
generated quasi-projective left $R$-module with $S=\End_R(M)$ and $N$ is an $M$-generated left $R$-module, then there
exists an order-preserving bijective correspondence between the sets of coclosed left $R$-submodules of $N$ and coclosed
left $S$-submodules of $\Hom_R(M,N)$. 
\end{abstract}

\thanks{The first author acknowledges the support of the grant PN-II-RU-TE-2011-3-0065. The fourth 
author acknowledges the support of the grant PN-II-RU-TE-2009-1 project ID\_303. Part
of the paper was carried out when the first author was visiting Gebze Institute of Technology in November 2010. He
gratefully acknowledges the support of TUBITAK and the kind hospitality of the host university.}


\maketitle

\section{Introduction}

An important general problem in module theory is to relate properties of a module with properties of its endomorphism
ring; or more generally, for a left $R$-module $M$ with $S=\End_R(M)$, to relate properties of a left $R$-module $N$
with properties of the left $S$-module $\Hom_R(M,N)$. Within the vast literature dealing with this problem, we point out
the work of J. Zelmanowitz \cite{Z}, which is closely related to our topic and motivates our study. He showed that if
$M$ is a left $R$-module with $S=\End_R(M)$ and $N$ is an $M$-faithful left $R$-module, then there exists an
order-preserving bijective correspondence between the sets of closed left $R$-submodules of $N$ and closed left
$S$-submodules of $\Hom_R(M,N)$ \cite[Theorem~1.2]{Z}.  

The aim of the present paper is to establish a result dual to that of J. Zelmanowitz, in terms of coclosed submodules.
While closed submodules coincide with complement submodules and every module has a complement, in general coclosed
submodules are different of supplement submodules and not every submodule has a supplement (see \cite{CLVW}). These are
some of the main obstacles in dualizing results on topics related to coclosed submodules; for instance compare the
theories of lifting modules \cite{CLVW} and extending modules \cite{DHSW}. In order to overcome these problems, we need
to find a suitable setting for our results and to use a more general approach. We shall first consider the context of
bounded lattices and we shall make use of Galois connections between them. The reason for doing
that is twofold: first, the notions involved are lattice-theoretic, and secondly, our approach clarifies the
exposition and gives a more natural explanation for certain conditions we have to impose on our modules. Our main
theorem shows that if $(A,\wedge,\vee,0,1)$ and $(B,\wedge,\vee,0,1)$ are two bounded lattices, and 
$(\alpha,\beta)$ is a special Galois connection (whose properties will be detailed later on), then $\alpha$ and
$\beta$ induce mutually inverse bijections between the set of coclosed elements $a\in A$ such that $\alpha(a)$ is
coclosed in $B$ and the set of coclosed elements $b\in B$ such that $\beta(b)$ has a unique coclosure in $A$. In
particular, under certain conditions, we obtain order-preserving mutually inverse bijections between the sets of
coclosed elements in $A$ and coclosed elements in $B$, which dualize and generalize the main theorem of J.
Zelmanowitz \cite[Theorem~1.2]{Z}. We apply our result to a particular Galois connection for modules, previously pointed
out by T. Albu and C. N\u ast\u asescu \cite{AN}, in order to deduce the following consequence: if $M$ is a finitely
generated quasi-projective left $R$-module with $S=\End_R(M)$ and $N$ is an $M$-generated left $R$-module, then there
exists an order-preserving bijective correspondence between the sets of coclosed left $R$-submodules of $N$ and coclosed
left $S$-submodules of $\Hom_R(M,N)$. We also relate the dual Goldie dimensions as well as the supplemented and the
lifting properties of the left $R$-module $N$ and the left $S$-module $\Hom_R(M,N)$.

\section{Cosmall Galois connections}

We shall make use of the concept of (monotone) Galois connection (e.g., see \cite{Erne}). This is defined on arbitrary
partially ordered sets, but we recall its definition on lattices, this being the setting in which we shall employ it. 

\begin{defn} \rm Let $(A,\leq)$ and $(B,\leq)$ be lattices. A {\it Galois connection} between them consists of a
pair $(\alpha, \beta)$ of two order-preserving functions $\alpha:A\to B$ and $\beta:B\to A$ such that for all $a\in
A$ and $b\in B$, we have $\alpha(a)\leq b$ if and only if $a\leq \beta(b)$. Equivalently, $(\alpha,\beta)$ is a Galois
connection if and only if for all $a\in A$, $a\leq \beta\alpha(a)$ and for all $b\in B$, $\alpha\beta(b)\leq b$. 

An element $a\in A$ (respectively $b\in B$) is called a \emph{Galois element} if $\beta\alpha(a)=a$ (respectively
$\alpha\beta(b)=b$). 
\end{defn}

As usual, one may view any lattice $(A,\leq)$ as a triple $(A,\wedge,\vee)$, where $\wedge$ and $\vee$ denote the
infimum and the supremum of elements in $A$. Recall that the lattice $A$ is bounded if it has a least element,
denoted by $0$, and a greatest element, denoted by $1$. If $A$ is bounded, then we tacitly assume that $0\neq 1$, and we
denote it as $(A,\wedge,\vee,0,1)$. For $a,a'\in A$, we also denote $[a,a']=\{x\in A\mid a\leq x\leq a'\}$.

We gather in the following lemma some well-known results (e.g., see \cite[Proposition~3.3]{AN}, \cite{Erne}), which
shall be used throughout the paper without further reference. 

\begin{lem} \label{l:galois} Let $(A,\leq)$ and $(B,\leq)$ be lattices, and $(\alpha,\beta)$ a Galois
connection, where $\alpha:A\to B$ and $\beta:B\to A$. Then: 

(i) $\alpha\beta\alpha=\alpha$ and $\beta\alpha\beta=\beta$.

(ii) $\alpha$ preserves all suprema in $A$ and $\beta$ preserves all infima in $B$. 

(iii) If $A$ and $B$ are bounded, then $\alpha(0)=0$ and $\beta(1)=1$. 

(iv) The restrictions of $\alpha$ and $\beta$ to the corresponding sets of Galois elements are mutually inverse
bijections.
\end{lem}

The module-theoretic concepts of direct summand, cosmall inclusion, coclosed submodule and coclosure of a submodule
(e.g., see \cite{CLVW}) have natural lattice counterparts.  

\begin{defn} \rm Let $(A,\wedge,\vee,0,1)$ be a bounded lattice. 

(1) An element $a\in A$ is called a {\it complement} in $A$ if there exists $a'\in A$ such that $a\wedge a'=0$ and
$a\vee a'=1$.

(2) Let $a,a'\in A$ be such that $a\leq a'$. Then $a'$ is called {\it cosmall} in $[a,1]$ if for any $x\in A$,
$1=a'\vee x$ implies $1=a\vee x$. 

(3) An element $a'\in A$ is called {\it coclosed} in $A$ if for any $a\in A$, $a'$ cosmall in $[a,1]$ implies $a=a'$.

(4) An element $a'\in A$ is called a {\it coclosure} of $a\in A$ in $A$ if $a$ is cosmall in $[a',1]$ and $a'$ is
coclosed
in $A$.
\end{defn}

The following lemma is well-known for submodule lattices, and its proof is straightforward.

\begin{lem} \label{l:cosmall} Let $(A,\wedge,\vee,0,1)$ be a bounded lattice. 

(i) Let $a,b,c\in A$ such that $a\leq b\leq c$. Then $c$ is cosmall in $[a,1]$ if and only if $b$ is
cosmall in $[a,1]$ and $c$ is cosmall in $[b,1]$.

(ii) If $A$ is modular, then every complement in $A$ is coclosed.
\end{lem}

Note that a coclosure of a given element might not exist. For instance, the subgroup $2\mathbb{Z}$ of the abelian group
$\mathbb{Z}$ has no coclosure in the subgroup lattice of $\mathbb{Z}$ (\cite[3.10]{CLVW}). 

We continue with an easy, but useful lemma. 

\begin{lem} \label{l:cocgal} Let $(A,\wedge,\vee,0,1)$ and $(B,\wedge,\vee,0,1)$ be bounded
lattices, and $(\alpha,\beta)$ a Galois connection, where $\alpha:A\to B$, $\alpha(1)=1$ and $\beta:B\to A$
preserves finite suprema. Then: 

(i) For all $b\in B$, $b$ is cosmall in $[\alpha\beta(b),1]$.

(ii) Every coclosed element of $B$ is Galois.  
\end{lem}

\begin{proof} (i) Let $b\in B$ and let $1=b\vee b'$ for some $b'\in B$. Then $1=\alpha\beta(1)=\alpha\beta(b)\vee
\alpha\beta(b')\leq \alpha\beta(b)\vee b'$, and so $1=\alpha\beta(b)\vee b'$. Hence $b$ is cosmall in
$[\alpha\beta(b),1]$. 

(ii) Clear by (i).
\end{proof}

Note that if $(\alpha,\beta)$ is a Galois connection between two bounded lattices and $\alpha,\beta$ are lattice
homomorphisms, then $\alpha(1)=1$ and $\beta$ preserves finite suprema.

We introduce two special types of Galois connection, which will be useful to us.

\begin{defn} \rm Let $(A,\wedge,\vee,0,1)$ and $(B,\wedge,\vee,0,1)$ be bounded lattices, and $(\alpha,\beta)$
a Galois connection, where $\alpha:A\to B$ and $\beta:B\to A$. We say that $(\alpha,\beta)$ is {\it cosmall} if for
all $a\in A$, $\beta\alpha(a)$ is cosmall in $[a,1]$. We say that a cosmall Galois connection $(\alpha,\beta)$ is {\it
UCC} if for every coclosed element $a\in A$, $a$ is the unique coclosure of $\beta\alpha(a)$ in $A$.
\end{defn}

We present some examples to illustrate the above theory. 

\begin{ex} \rm \label{e:ex} Consider the abelian groups $G=\mathbb{Z}_p\times \mathbb{Z}_{q^2}$ for some primes $p$ and
$q$ with $p\neq q$, and $G'=\mathbb{Z}_2\times \mathbb{Z}_4$, where $\mathbb{Z}_n$ denotes the cyclic group of order
$n\in \mathbb{N}$. Their subgroup lattices $S(G)$ and $S(G')$ have the following forms respectively: 
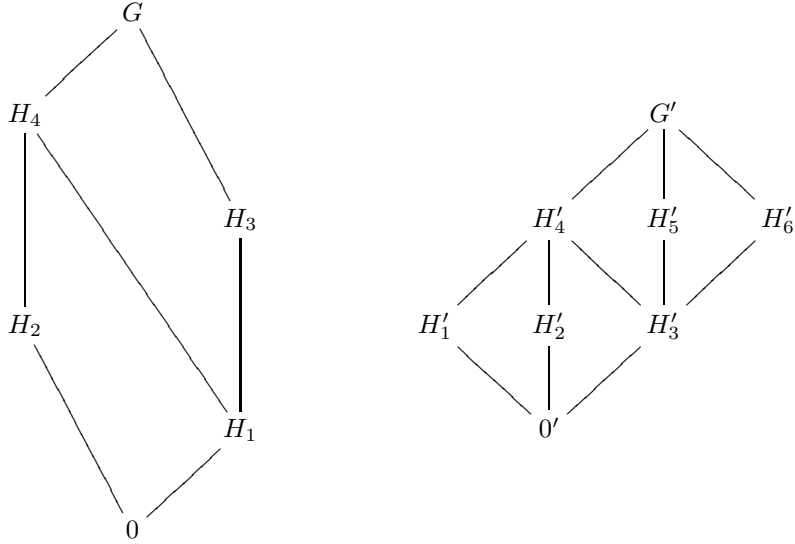
\begin{small} 
\[\SelectTips{cm}{}
\xymatrix{
 & G \ar@{-}[dl] \ar@{-}[ddr] & &&&&& \\ 
H_4 \ar@{-}[dd] \ar@{-}[dddrr] & &  & &&& G' \ar@{-}[dl] \ar@{-}[d] \ar@{-}[dr] &\\
 & & H_3 \ar@{-}[dd] &&& H_4' \ar@{-}[dl] \ar@{-}[d] \ar@{-}[dr] & H_5' \ar@{-}[d] & H_6' \ar@{-}[dl]\\
H_2 \ar@{-}[ddr] & &&& H_1' \ar@{-}[dr] & H_2' \ar@{-}[d] & H_3' \ar@{-}[dl] \\
 & & H_1 \ar@{-}[dl] &&& 0' & &    \\
& 0 &  &&&&  
}\]
\end{small} 

It is easy to check that $H_1$ is cosmall in $[0,G]$, $H_4$ is cosmall in $[H_2,G]$, $H_1$ and $H_4$ are the only
subgroups of $G$ which are not coclosed, $0$ is a coclosure of $H_1$ in $G$, and $H_2$ is a coclosure of $H_4$ in $G$.
Also, $H_3'$ is cosmall in $[0',G']$, $H_4'$ is cosmall in $[H_1',G']$ and $[H_2',G']$, $H_3'$ and $H_4'$ are the
only subgroups of $G'$ which are not coclosed, $0'$ is a coclosure of $H_3'$ in $G'$, and $H_1',H_2'$ are coclosures of
$H_4'$ in $G'$. For properties of subgroups with unique (co)closure and cosmall subgroups of abelian groups the
reader is referred to \cite{CO} and \cite{CS}.

(1) Consider the functions $\alpha:S(G)\to S(G)$ defined by $\alpha(0)=\alpha(H_1)=0$, $\alpha(H_2)=\alpha(H_4)=H_3$,
$\alpha(H_3)=H_2$ and $\alpha(G)=G$, $\beta:S(G)\to S(G)$ defined by $\beta(0)=\beta(H_1)=H_1$,
$\beta(H_2)=\beta(H_4)=H_3$, $\beta(H_3)=H_4$, $\beta(G)=G$. Then $(\alpha,\beta)$ is a Galois connection from the
lattice $(S(G),\subseteq)$ to itself. For every $H\in S(G)\setminus \{H_1,H_4\}$ we have $\alpha\beta(H)=H$. Hence $0$, 
$H_2$, $H_3$ and $G$ are Galois elements in the codomain $B=S(G)$ of $\alpha$. Also, $H_1$ is cosmall in
$[\alpha\beta(H_1),G]=[0,G]$ and $H_4$ is cosmall in $[\alpha\beta(H_4),G]=[H_2,G]$. On the other hand, for every $H\in
S(G)\setminus \{0,H_2\}$ we have $\beta\alpha(H)=H$. Hence $H_1$, $H_3$, $H_4$ and $G$ are Galois elements in
the domain $A=S(G)$ of $\alpha$. Also, $\beta\alpha(0)=H_1$ is cosmall in $[0,G]$ and $\beta\alpha(H_2)=H_4$ is cosmall
in $[H_2,G]$. Moreover, for every coclosed $H\in S(G)$, $H$ is the unique coclosure of $\beta\alpha(H)$ in $S(G)$. Hence
$(\alpha,\beta)$ is a UCC cosmall Galois connection. Note that $H_2$ is a coclosed element, but not a Galois element in
the domain $A=S(G)$ of $\alpha$. Hence not every coclosed element of $A$ is Galois under the hypotheses of Lemma
\ref{l:cocgal}.

(2) Consider the functions $\alpha:S(G)\to S(G)$ defined by $\alpha(H_2)=H_4$,
$\alpha(H_3)=G$, and $\alpha(H)=H$ for every $H\in S(G)\setminus \{H_2,H_3\}$, and $\beta:S(G)\to S(G)$ defined by
$\beta(H_2)=0$, $\beta(H_3)=H_1$ and $\beta(H)=H$ for every $H\in S(G)\setminus \{H_2,H_3\}$. Then $(\alpha,\beta)$ is
a Galois connection from the lattice $(S(G),\subseteq)$ to itself. But $(\alpha,\beta)$ is not a cosmall Galois
connection, because for instance $\beta\alpha(H_3)=G$ is not cosmall in $[H_3,G]$. Using the same setting, let
us also show that the hypothesis on $\beta$ to preserve finite suprema in Lemma \ref{l:cocgal} cannot be removed. Note
that we have $\beta(H_3+H_4)=G\neq H_4=\beta(H_3)+\beta(H_4)$. Then $H_2$ is not cosmall in
$[\alpha\beta(H_2),G]=[0,G]$. Also, $H_4$ is a Galois element in the codomain $B=S(G)$ of $\alpha$, but not a coclosed
element. 

(3) Consider the functions $\alpha':S(G')\to S(G')$ defined by $\alpha'(H_3')=0$, $\alpha(H_4')=H_1'$ and
$\alpha(H')=H'$ for every $H'\in S(G')\setminus \{H_3',H_4'\}$, and $\beta':S(G')\to S(G')$ defined by
$\beta'(0')=H_3'$, $\beta'(H_1')=H_4'$ and $\beta(H')=H'$ for every $H'\in S(G')\setminus \{0',H_1'\}$. Then
$(\alpha',\beta')$ is a Galois connection from the lattice $(S(G'),\subseteq)$ to itself. For every $H'\in
S(G')\setminus \{H_3',H_4'\}$ we have $\alpha\beta(H')=H'$. Also, $H_3'$ is cosmall in
$[\alpha'\beta'(H_3'),G']=[0',G']$ and $H_4'$ is cosmall in $[\alpha'\beta'(H_4'),G']=[H_1',G']$. On the other hand, for
every $H'\in S(G')\setminus \{0',H_1'\}$ we have $\beta'\alpha'(H')=H'$. Also, $\beta'\alpha'(0')=H_3'$ is cosmall in
$[0',G']$ and $\beta'\alpha'(H_1')=H_4'$ is cosmall in $[H_1',G']$. Hence $(\alpha',\beta')$ is a cosmall Galois
connection. But it is not UCC, because for the coclosed subgroup $H_1'$, $\beta'\alpha'(H_1')$ has two coclosures in
$G'$, namely $H_1'$ and $H_2'$.  
\end{ex}

We end this section with some results which show that the cosmall property and the dual Goldie dimension may be
transferred through cosmall Galois connections.

\begin{lem} \label{l:transfer} Let $(A,\wedge,\vee,0,1)$ and $(B,\wedge,\vee,0,1)$ be bounded lattices, and  
$(\alpha,\beta)$ a cosmall Galois connection, where $\alpha:A\to B$, $\alpha(1)=1$ and $\beta:B\to A$ preserves
finite suprema. 

(i) Let $a,a'\in A$. Then $a'$ is cosmall in $[a,1]$ if and only if $\alpha(a')$ is cosmall in $[\alpha(a),1]$.

(ii) Let $b,b'\in B$. Then $b'$ is cosmall in $[b,1]$ if and only if $\beta(b')$ is cosmall in $[\beta(b),1]$.
\end{lem}

\begin{proof} (i) Assume that $a'$ is cosmall in $[a,1]$. Let $1=\alpha(a')\vee b'$ for some $b'\in B$. Then
$1=\beta(1)=\beta\alpha(a')\vee \beta(b')$. Since $(\alpha,\beta)$ is a cosmall Galois connection, it follows that
$1=a'\vee \beta(b')$. Now by hypothesis, we must have $1=a\vee \beta(b')$. Then $1=\alpha(1)=\alpha(a)\vee
\alpha\beta(b')\leq \alpha(a)\vee b'$, and so $1=\alpha(a)\vee b'$. Hence $\alpha(a')$ is cosmall in $[\alpha(a),1]$.

Conversely, assume that $\alpha(a')$ is cosmall in $[\alpha(a),1]$. Let $1=a'\vee a''$ for some $a''\in A$. Then
$1=\alpha(1)=\alpha(a')\vee \alpha(a'')$, whence $1=\alpha(a)\vee \alpha(a'')$ by hypothesis. Then
$1=\beta(1)=\beta\alpha(a)\vee \beta\alpha(a'')$. Since $(\alpha,\beta)$ is a cosmall Galois connection, it follows
that $1=a\vee a''$. Hence $a'$ is cosmall in $[a,1]$. 

(ii) Assume that $b'$ is cosmall in $[b,1]$. Let $1=\beta(b')\vee a'$ for some $a'\in A$. Then
$1=\alpha(1)=\alpha\beta(b')\vee \alpha(a')\leq b'\vee \alpha(a')$, hence $1=b'\vee\alpha(a')$, and by hypothesis
$1=b\vee\alpha(a')$. Then $1=\beta(1)=\beta(b)\vee \beta\alpha(a')$. Since $(\alpha,\beta)$ is a cosmall Galois
connection, it follows that $1=\beta(b)\vee a'$. Hence $\beta(b')$ is cosmall in $[\beta(b),1]$.

Conversely, assume that $\beta(b')$ is cosmall in $[\beta(b),1]$. Let $1=b'\vee b''$ for some $b''\in B$. Then
$1=\beta(1)=\beta(b')\vee \beta(b'')$, whence $1=\beta(b)\vee \beta(b'')$ by hypothesis. Then
$1=\alpha(1)=\alpha\beta(b)\vee \alpha\beta(b'')\leq b\vee b''$, and so $1=b\vee b''$. Hence $b'$ is cosmall in $[b,1]$.
\end{proof}

Let $(X,\wedge,\vee,0,1)$ be a bounded modular lattice. Recall that a subset $Y$ of $X\setminus \{1\}$ is called
\emph{meet-independent} if $(y_1\wedge \ldots \wedge y_n)\vee x=1$ for every finite subset $\{y_1,\dots,y_n\}$ of $Y$
and every $x\in Y\setminus \{y_1,\dots,y_n\}$. If there is a finite supremum $d$ of all numbers $k$ such that $X$
has a meet-independent subset with $k$ elements, then $X$ \emph{has dual Goldie dimension} (or \emph{hollow
dimension}) $d$; otherwise $X$ has infinite dual Goldie dimension (see \cite[Theorem~9]{GP}). We denote the dual Goldie
dimension of $X$ by ${\rm hdim}(X)$. 

\begin{cor} \label{c:dualg} Let $(A,\wedge,\vee,0,1)$ and $(B,\wedge,\vee,0,1)$ be bounded modular lattices, and  
$(\alpha,\beta)$ a cosmall Galois connection, where $\alpha:A\to B$, $\alpha(1)=1$ and $\beta:B\to A$ preserves
finite suprema. Then: 

(i) ${\rm hdim}(A)\leq {\rm hdim}(B)$. 

(ii) If every element of $A$ is Galois, then ${\rm hdim}(A)={\rm hdim}(B)$.
\end{cor}

\begin{proof} (i) We prove that if $A$ has a finite meet-independent subset with $m$ elements, then so has $B$.
This will imply that ${\rm hdim}(A)\leq {\rm hdim}(B)$ in both finite and infinite cases for ${\rm hdim}(A)$. To this
end, let $\{a_1,\dots,a_m\}$ be a meet-independent subset of $A$. We claim that $Y=\{\alpha(a_1),\dots,\alpha(a_m)\}$ is
a meet-independent subset of $B$ with $m$ elements. 

First, we point out that $\alpha(a_i)\neq 1$ for every $i\in \{1,\dots,m\}$. Indeed, if there is $i\in
\{1,\dots,m\}$ such that $\alpha(a_i)=1$, then $\beta\alpha(a_i)=1$. Since $\beta\alpha(a_i)$ is cosmall in $[a_i,1]$,
we must have $a_i=1$, contradiction.

Secondly, suppose that $\alpha(a_i)=\alpha(a_j)$ for some $i,j\in \{1,\dots,m\}$ with $i\neq j$. Since $a_i\vee a_j=1$,
we have $\alpha(a_i)\vee \alpha(a_j)=\alpha(a_i\vee a_j)=\alpha(1)=1$. We conclude that $\alpha(a_i)=\alpha(a_j)=1$,
which contradicts the previous point. Hence $Y$ has $m$ elements.

Finally, the meet-independence of $\{a_1,\dots,a_m\}$ is known to be equivalent to the condition $(a_1\wedge \ldots
\wedge a_{k-1})\vee a_k=1$ for every $k\in \{2,\dots,m\}$. Let $k\in \{2,\dots,m\}$. Then $\alpha(a_1\wedge \ldots
\wedge a_{k-1})\vee \alpha(a_k)=1$. Since $\alpha(a_1\wedge \ldots \wedge a_{k-1})\leq \alpha(a_1)\wedge \ldots \wedge
\alpha(a_{k-1})$, it follows that $(\alpha(a_1)\wedge \ldots \wedge \alpha(a_{k-1}))\vee \alpha(a_k)=1$. This shows that
$Y$ is meet-independent.

(ii) Suppose that every element of $A$ is Galois. When $A$ has infinite dual Goldie dimension, the conclusion is clear
by (i). Assume that ${\rm hdim}(A)=m$. By \cite[Theorem~9]{GP}, $A$ has a meet-independent
subset $\{a_1,\dots,a_m\}$ such that $a_1\wedge \ldots \wedge a_m$ is cosmall in $[0,1]$. Now assume that there is a
meet-independent subset $Y$ of $B$ with more than $m$ elements, possibly infinite. Then
there is a meet-independent subset $\{b_1,\dots,b_{m+1}\}\subseteq Y$ of $B$. By dualizing an argument from the proof of
\cite[Theorem~5]{GP}, $\{\alpha(a_1),\dots,\alpha(a_m),b_{m+1}\}$ is also a meet-independent subset of $B$. Then we have
$(\alpha(a_1)\wedge \ldots \wedge \alpha(a_m))\vee b_{m+1}=1$. Since every element of $A$ is Galois, it follows that
$(a_1\wedge \ldots \wedge a_m)\vee \beta(b_{m+1})=(\beta\alpha(a_1)\wedge \ldots \wedge \beta\alpha(a_m))\vee
\beta(b_{m+1})=\beta((\alpha(a_1)\wedge \ldots \wedge \alpha(a_m))\vee b_{m+1})=\beta(1)=1$. Since $a_1\wedge \ldots
\wedge a_m$ is cosmall in $[0,1]$, we deduce that $\beta(b_{m+1})=1$. Then $1=\alpha\beta(b_{m+1})\leq b_{m+1}$, and so
$b_{m+1}=1$, which is a contradiction. Therefore, by (i) and the above, ${\rm hdim}(B)=m={\rm hdim}(A)$. 
\end{proof}

\section{Correspondences}

Let $(A,\leq)$ and $(B,\leq)$ be two lattices, and let $(\alpha,\beta)$ be a Galois connection, where $\alpha:A\to
B$ and $\beta:B\to A$. We have seen in Lemma \ref{l:cocgal} and Example \ref{e:ex} (2) that the set of coclosed
elements in $B$ is in general strictly included in the set of Galois elements in $B$. Also, we have already seen that
the restrictions of $\alpha$ and $\beta$ to the corresponding sets of Galois elements are mutually inverse bijections.
We shall show that, under certain conditions, these bijections restrict to ones between the corresponding sets of
complement Galois elements, or between the corresponding sets of coclosed Galois elements. 

\begin{lem} \label{l:compl} Let $(A,\wedge,\vee,0,1)$ and $(B,\wedge,\vee,0,1)$ be bounded lattices, and 
$(\alpha,\beta)$ a Galois connection, where $\alpha:A\to B$, $\alpha(1)=1$ and $\beta:B\to A$ preserves finite
suprema. Assume that all complements in $A$ are Galois elements. Then:

(i) If $\beta$ is injective, then $\alpha$ preserves complement Galois elements.

(ii) $\beta$ preserves complement Galois elements.

(iii) If $\beta$ is injective, then $\alpha$ and $\beta$ restrict to order-preserving mutually inverse bijections
between the sets of complement Galois elements in $A$ and $B$.
\end{lem}

\begin{proof} First note that, since $0\in A$ is a Galois element, we have $\beta\alpha(0)=0$, and so $\beta(0)=0$.

(i) Assume that $\beta$ is injective. Let $a$ be a complement (Galois) element in $A$. Then there is
$a'\in A$ such that $a\wedge a'=0$ and $a\vee a'=1$. Hence $\alpha(a)\vee \alpha(a')=\alpha(1)=1$. Also, we have
$0=a\wedge a'=\beta\alpha(a)\wedge \beta\alpha(a')=\beta(\alpha(a)\wedge \alpha(a'))$, whence $\alpha(a)\wedge
\alpha(a')=0$ by the injectivity of $\beta$. Thus $\alpha(a)$ is a complement in $B$. By Lemma \ref{l:galois},
$\alpha(a)$ is a Galois element in $B$. 

(ii) Let $b$ be a complement Galois element in $B$. Then there is $b'\in B$ such that $b\wedge b'=0$ and $b\vee b'=1$.
Then $\beta(b)\wedge \beta(b')=\beta(0)=0$ and $\beta(b)\vee \beta(b')=\beta(1)=1$. Thus $\beta(b)$ is a complement in
$A$. By Lemma \ref{l:galois}, $\beta(b)$ is a Galois element in $A$.

(iii) Clear by (i), (ii) and Lemma \ref{l:galois}.
\end{proof}

Next we need to recall the following notions, which are the lattice-theoretic versions of the corresponding ones
for modules (e.g., see \cite{CLVW}).

\begin{defn} \label{d:lifting} \rm Let $(A,\wedge,\vee,0,1)$ be a bounded lattice. 

(1) An element $a\in A$ is called a {\it supplement} if it is a supplement of some $a'\in A$, that is, $a$ is
minimal in $A$ with the property that $1=a\vee a'$.

(2) $A$ is called {\it supplemented} if every $a\in A$ has a supplement in $A$. 

(3) $A$ is called {\it amply supplemented} if for every $a\in A$ there exists a supplement $x\in A$ such that $a$ is
cosmall in $[x,1]$.

(4) $A$ is called {\it UCC} if every $a\in A$ has a unique coclosure in $A$, which will be denoted by $\bar a$. 

(5) $A$ is called {\it lifting} if for every $a\in A$ there exists a complement $x\in A$ such that
$a$ is cosmall in $[x,1]$.
\end{defn}

The following lemma has a similar proof as for modules (see \cite[Corollary~3.8]{GV},
\cite[Lemma~1.1 and Proposition~1.5]{Keskin} and \cite[Chapter~41]{Wis}). 

\begin{lem} \label{l:suppl} Let $(A,\wedge,\vee,0,1)$ be a bounded modular lattice. 

(i) If $a\in A$ is a supplement, then $a$ is coclosed.

(ii) If $A$ is supplemented and $a\in A$ is coclosed, then $a$ is a supplement. 

(iii) If $A$ is amply supplemented, then every element $a\in A$ has a coclosure in $A$.  

(iv) If $A$ is amply supplemented UCC and $a_1,a_2\in A$ are such that $a_1\leq a_2$, then their coclosures satisfy 
$\overline{a_1}\leq \overline{a_2}$.  
\end{lem}

Now we may relate supplemented and lifting type properties of some bounded modular lattices. 

\begin{thm} \label{t:lifting} Let $(A,\wedge,\vee,0,1)$ and $(B,\wedge,\vee,0,1)$ be bounded modular lattices and
$(\alpha,\beta)$ a Galois connection, where $\alpha:A\to B$, $\alpha(1)=1$ and $\beta:B\to A$ preserves
finite suprema. Assume that every element of $A$ is Galois. Then:

(i) $A$ is supplemented if and only if $B$ is supplemented. 

(ii) If $B$ is lifting, then $A$ is lifting. Conversely, if $\beta$ is injective and $A$ is lifting, then $B$ is
lifting.
\end{thm}

\begin{proof} (i) First assume that $A$ is supplemented. Let $b\in B$. Then $\beta(b)$ has a supplement $x$ in $A$.
Hence $1=\beta(b)\vee x$, which implies that $1=\alpha(1)=\alpha\beta(b)\vee \alpha(x)\leq b\vee \alpha(x)$, and so
$1=b\vee \alpha(x)$. Now let $1=b\vee y'$ for some $y'\in B$ with $y'\leq \alpha(x)$. Then $1=\beta(1)=\beta(b)\vee
\beta(y')$. Since $x$ is Galois, we have $\beta(y')\leq \beta\alpha(x)=x$, whence $\beta(y')=x$ by
the minimality of $x$. It follows that $\alpha(x)=\alpha\beta(y')\leq y'$, and so $y'=\alpha(x)$.
This shows the minimality of the supplement $\alpha(x)$ of $b$ in $B$. Thus $B$ is supplemented. 

Conversely, assume that $B$ is supplemented, and let $a\in A$. Then $\alpha(a)$ has a supplement
$y$ in $B$, and so $1=\alpha(a)\vee y$. By hypothesis it follows that $1=\beta(1)=\beta\alpha(a)\vee
\beta(y)=a\vee \beta(y)$. Now let $1=a\vee x'$ for some $x'\in A$ with
$x'\leq \beta(y)$. Then $1=\alpha(1)=\alpha(a)\vee \alpha(x')$ and $\alpha(x')\leq \alpha\beta(y)\leq y$. By the
minimality of $y$, $\alpha(x')=y$. Since $x'$ is Galois, we have $x'=\beta\alpha(x')=\beta(y)$. This
shows the minimality of the supplement $\beta(y)$ of $a$ in $A$. Thus $A$ is supplemented.
 
(ii) Assume that $B$ is lifting. Let $a\in A$. Then there exists a complement $b'\in B$ such that $\alpha(a)$ is
cosmall in $[b',1]$. Since $a$ is Galois, $a=\beta\alpha(a)$ is cosmall
in $[\beta(b'),1]$ by Lemma \ref{l:transfer}. By Lemmas \ref{l:cosmall} and \ref{l:cocgal}, $b'$ is a Galois element in
$B$. Then by Lemma \ref{l:compl}, $\beta(b')$ is a complement in $A$. It follows that $A$ is lifting. 

Now assume that $\beta$ is injective and $A$ is lifting. Let $b\in B$. Then there exists a complement $a'\in A$ such
that $\beta(b)$ is cosmall in $[a',1]$. By Lemma \ref{l:transfer}, $\alpha\beta(b)$ is cosmall in $[\alpha(a'),1]$. By 
Lemma \ref{l:cocgal}, $b$ is cosmall in $[\alpha\beta(b),1]$. Hence by Lemma \ref{l:cosmall}, $b$ is cosmall in
$[\alpha(a'),1]$. By Lemma \ref{l:compl}, $\alpha(a')$ is a complement in $B$. It follows that $B$ is lifting.
\end{proof}

Next we establish our general theorem on a bijective correspondence between sets of coclosed elements induced by some
special Galois connections. In order to obtain it, it is natural to try to define some maps by means of unique
coclosures of elements, when they do exist. We see in the following theorem that the Galois connection
and the condition that these maps are well-defined create a slightly asymmetric situation. 

\begin{thm} \label{t:main} Let $(A,\wedge,\vee,0,1)$ and $(B,\wedge,\vee,0,1)$ be bounded lattices and 
$(\alpha,\beta)$ a UCC cosmall Galois connection, where $\alpha:A\to B$, $\alpha(1)=1$ and $\beta:B\to A$ preserves
finite suprema. For $x\in A\cup B$ denote by $\bar{x}$ the unique coclosure of $x$, when it does exist. 

(i) Denote
\begin{align*} \mathcal{A}&=\{a\in A\mid a \textrm{ is coclosed in $A$ and $\alpha(a)$ has a unique coclosure in
$B$}\}, \\ 
\mathcal{B}&=\{b\in B\mid b \textrm{ is coclosed in $B$ and $\beta(b)$ has a unique coclosure in $A$}\}.
\end{align*}
Consider $\varphi:\mathcal{A}\to \mathcal{B}$ defined by $\varphi(a)=\overline{\alpha(a)}$ for every $a\in \mathcal{A}$,
and $\psi:\mathcal{B}\to \mathcal{A}$ defined by $\psi(b)=\overline{\beta(b)}$ for every $b\in \mathcal{B}$.
Then $\varphi$ is a well-defined map if and only if $\mathcal{A}$ coincides with the set of coclosed elements $a\in
A$ such that $\alpha(a)$ is coclosed in $B$. Also, $\psi$ is a well-defined map. 

(ii) Denote
\begin{align*} \mathcal{A}&=\{a\in A\mid a \textrm{ is coclosed in $A$ and $\alpha(a)$ is coclosed in $B$}\}, \\
\mathcal{B}&=\{b\in B\mid b \textrm{ is coclosed in $B$ and $\beta(b)$ has a unique coclosure in $A$}\}.
\end{align*}
Then the maps $\varphi:\mathcal{A}\to \mathcal{B}$ defined by $\varphi(a)=\alpha(a)$ for every $a\in \mathcal{A}$, 
and $\psi:\mathcal{B}\to \mathcal{A}$ defined by $\psi(b)=\overline{\beta(b)}$ for every $b\in \mathcal{B}$,
are mutually inverse bijections.  
\end{thm}

\begin{proof} (i) First assume that $\varphi$ is well-defined. Let $a\in \mathcal{A}$. Then
$\varphi(a)=\overline{\alpha(a)}\in \mathcal{B}$, and so $\beta(\overline{\alpha(a)})$ has a unique coclosure,
say $a_0$. Then $\beta(\overline{\alpha(a)})$ is cosmall in $[a_0,1]$. Since $\alpha(a)$ is cosmall in
$[\overline{\alpha(a)},1]$, $\beta\alpha(a)$ is cosmall in $[\beta(\overline{\alpha(a)}),1]$ by Lemma \ref{l:transfer}.
Then by Lemma \ref{l:cosmall} $\beta\alpha(a)$ is cosmall in $[a_0,1]$. Since $(\alpha,\beta)$ is UCC cosmall, $a$ is
the unique coclosure of $\beta\alpha(a)$, hence $a_0=a$. Then $\beta(\overline{\alpha(a)})$ is cosmall in $[a,1]$. By
Lemmas \ref{l:cocgal} and \ref{l:transfer}, $\overline{\alpha(a)}=\alpha\beta(\overline{\alpha(a)})$ is cosmall in
$[\alpha(a),1]$. Hence $\overline{\alpha(a)}=\alpha(a)$, and so $\alpha(a)$ is coclosed in $B$.  

Now assume that $\mathcal{A}$ coincides with the set of coclosed elements $a\in A$ such that $\alpha(a)$ is coclosed in
$B$. Then $\varphi(a)=\overline{\alpha(a)}=\alpha(a)$ for every $a\in \mathcal{A}$. Let $a\in \mathcal{A}$. Then
$\alpha(a)$ is coclosed in $B$ by hypothesis. Since $(\alpha,\beta)$ is UCC cosmall, $\beta\alpha(a)$ has a unique
coclosure in $A$. Hence $\varphi(a)=\alpha(a)\in \mathcal{B}$, and so $\varphi$ is well-defined. 

In order to show that $\psi$ is well-defined, let $b\in \mathcal{B}$. Then $\overline{\beta(b)}$ is coclosed in $A$.
Since $\beta(b)$ is cosmall in $[\overline{\beta(b)},1]$, $\alpha\beta(b)$ is cosmall in
$[\alpha(\overline{\beta(b)}),1]$ by Lemma \ref{l:transfer}. Since $b\in B$ is coclosed, we have $\alpha\beta(b)=b$ by
Lemma \ref{l:cocgal}. Then $b$ is cosmall in $[\alpha(\overline{\beta(b)}),1]$. It follows that
$\alpha(\overline{\beta(b)})=b$, because $b\in B$ is coclosed. Hence $\alpha(\overline{\beta(b)})$ is coclosed in $B$.
Thus $\psi(b)=\overline{\beta(b)}\in \mathcal{A}$. Note that $\psi$ is also well-defined if the codomain is the set of
coclosed elements $a\in A$ such that $\alpha(a)$ is coclosed in $B$. 

(ii) The maps $\varphi$ and $\psi$ are well-defined by (i). If $a\in \mathcal{A}$, then $a$ is the unique coclosure of
$\beta\alpha(a)$ in $A$ because $(\alpha,\beta)$ is UCC cosmall, and we have
$\psi\varphi(a)=\overline{\beta\alpha(a)}=a$. If $b\in \mathcal{B}$, then we have
$\varphi\psi(b)=\alpha(\overline{\beta(b)})=b$ as above. Therefore, $\varphi$ and $\psi$ are mutually inverse
bijections. 
\end{proof}

We apply Theorem \ref{t:main} in two relevant situations as follows. The first one, when every element in $A$ is Galois,
will be particularly considered in the last section of this paper. The second one, when $A$ is amply supplemented
modular and $(\alpha,\beta)$ is UCC cosmall, may be applied to any cosmall Galois connection $(\alpha,\beta)$ between
finite UCC abelian groups (e.g., see \cite{CO}, \cite{CS}). 

\begin{thm} \label{t:main2} Let $(A,\wedge,\vee,0,1)$ and $(B,\wedge,\vee,0,1)$ be bounded lattices and 
$(\alpha,\beta)$ a Galois connection, where $\alpha:A\to B$, $\alpha(1)=1$ and $\beta:B\to A$ preserves
finite suprema. Then there are mutually inverse bijections between the sets $\mathcal{C}_A$ of coclosed
elements in $A$ and $\mathcal{C}_B$ of coclosed elements in $B$ provided one of the following conditions holds:

(i) Every element in $A$ is Galois. 

(ii) $A$ is amply supplemented modular and $(\alpha,\beta)$ is UCC cosmall. 

If either every element in $A$ is Galois, or $A$ is amply supplemented modular UCC and $(\alpha,\beta)$ is cosmall,
then the above bijections are order-preserving.
\end{thm}

\begin{proof} In both cases we use the notation from Theorem \ref{t:main}, and prove that $\mathcal{A}=\mathcal{C}_A$
and $\mathcal{B}=\mathcal{C}_B$. Then by Theorem \ref{t:main} the required mutually inverse bijections will be given by
$\varphi$ and $\psi$. Clearly, $\mathcal{A}\subseteq \mathcal{C}_A$ and $\mathcal{B}\subseteq \mathcal{C}_B$. 

(i) Assume that every element in $A$ is Galois. Then $(\alpha,\beta)$ is clearly UCC cosmall. 

Let $a\in \mathcal{C}_A$. By hypothesis, we have $a=\beta\alpha(a)$. We claim
that $\alpha(a)$ is coclosed in $B$. To this end, let $b'\in B$ be such that $\alpha(a)$ is cosmall in $[b',1]$. By
Lemma \ref{l:transfer}, $a=\beta\alpha(a)$ is cosmall in $[\beta(b'),1]$. Since $a$ is coclosed in $A$, we must have
$\beta(b')=a$. It follows that $\alpha(a)=\alpha\beta(b')\leq b'$, whence $b'=\alpha(a)$. Thus $\alpha(a)$ is
coclosed in $B$, and so $a\in \mathcal{A}$.   

Now let $b\in \mathcal{C}_B$. We claim that $\beta(b)$ is coclosed in $A$. To this end, let $a'\in
A$ be such that $\beta(b)$ is cosmall in $[a',1]$. By Lemma \ref{l:transfer}, $\alpha\beta(b)$ is cosmall in
$[\alpha(a'),1]$. By Lemma \ref{l:cocgal}, every coclosed element in $B$ is a Galois element, hence
$b=\alpha\beta(b)$, and so $b$ is cosmall in $[\alpha(a'),1]$. Since $b$ is coclosed in $B$, it follows that
$\alpha(a')=b$. Then by hypothesis we have $a'=\beta\alpha(a')=\beta(b)$. Thus $\beta(b)$ is coclosed in $A$, and so
$b\in \mathcal{B}$. 

It follows that $\mathcal{A}=\mathcal{C}_A$ and $\mathcal{B}=\mathcal{C}_B$. By Theorem \ref{t:main}, the maps
$\varphi=\alpha$ and $\psi=\beta$ are mutually inverse bijections between $\mathcal{C}_A$ and $\mathcal{C}_B$. 

(ii) Assume that $A$ is amply supplemented modular and $(\alpha,\beta)$ is UCC cosmall.  

Let $a\in \mathcal{C}_A$. We claim that $\alpha(a)$ is coclosed in $B$. To this end, let $b'\in B$ be such that
$\alpha(a)$ is cosmall in $[b',1]$. By Lemma \ref{l:transfer}, $\beta\alpha(a)$ is cosmall in $[\beta(b'),1]$.
Since $A$ is amply supplemented modular, Lemma \ref{l:suppl} yields a coclosure of $\beta(b')$ in $A$, say $a_0$. Then
$\beta(b')$ is cosmall in $[a_0,1]$, whence $\beta\alpha(a)$ is cosmall in $[a_0,1]$ by Lemma \ref{l:cosmall}. Hence
$a_0$ is a coclosure of $\beta\alpha(a)$ in $A$. Since $(\alpha,\beta)$ is UCC cosmall, $a$ is the unique coclosure of
$\beta\alpha(a)$ in $A$, hence $a_0=a$. Then $a\leq \beta(b')$, which implies $\alpha(a)\leq b'$ because
$(\alpha,\beta)$ is a Galois connection. Hence $\alpha(a)=b'$, and so $\alpha(a)$ is coclosed in $B$. Thus $a\in
\mathcal{A}$.

Now let $b\in \mathcal{C}_B$. We claim that $\beta(b)$ has a unique coclosure in $A$. The existence of a coclosure of
$\beta(b)$ in $A$ follows by Lemma \ref{l:suppl}, because $A$ is amply supplemented modular. Now assume that $a_1$ and
$a_2$ are two coclosures of $\beta(b)$ in $A$. Then $\beta(b)$ is coclosed both in $[a_1,1]$ and in $[a_2,1]$. By
Lemma \ref{l:transfer}, $\alpha\beta(b)$ is cosmall both in $[\alpha(a_1),1]$ and in $[\alpha(a_2),1]$. Since $b$ is
coclosed, we have $\alpha\beta(b)=b$ by Lemma \ref{l:cocgal}. Hence $b$ is cosmall both in $[\alpha(a_1),1]$ and in
$[\alpha(a_2),1]$. But $b$ is coclosed in $B$, hence we must have $\alpha(a_1)=\alpha(a_2)=b$. Since
$(\alpha,\beta)$ is a cosmall Galois connection, $\beta\alpha(a_1)=\beta\alpha(a_2)$ is cosmall in $[a_1,1]$. Then
$a_1$ is a coclosure of $\beta\alpha(a_2)$ in $A$. Since $(\alpha,\beta)$ is UCC cosmall, $a_2$ is the unique coclosure
of $\beta\alpha(a_2)$ in $A$, hence $a_1=a_2$. Then $\beta(b)$ has a unique coclosure in $A$. Thus $b\in \mathcal{B}$. 

It follows that $\mathcal{A}=\mathcal{C}_A$ and $\mathcal{B}=\mathcal{C}_B$. By Theorem \ref{t:main}, the maps
$\varphi$ and $\psi$ are mutually inverse bijections between $\mathcal{C}_A$ and $\mathcal{C}_B$. 

Finally, if every element in $A$ is Galois, then $\varphi$ and $\psi$ are order-preserving, because so are $\alpha$ and
$\beta$. Also, if $A$ is amply supplemented modular UCC and $(\alpha,\beta)$ is cosmall, then $\varphi$ and $\psi$
are order-preserving by hypothesis and Lemma \ref{l:suppl}.  
\end{proof}

\begin{cor} \label{c:bij} Let $(A,\wedge,\vee,0,1)$ and $(B,\wedge,\vee,0,1)$ be bounded lattices and 
$(\alpha,\beta)$ a UCC cosmall Galois connection, where $\alpha:A\to B$, $\alpha(1)=1$, $\beta:B\to A$ preserves
finite suprema, and $A$ is amply supplemented modular. Then the following are equivalent:

(i) The mutually inverse bijections from Theorem \ref{t:main2} are the restrictions of $\alpha$ and $\beta$ to
the sets $\mathcal{C}_A$ of coclosed elements in $A$ and $\mathcal{C}_B$ of coclosed elements in $B$ respectively.

(ii) Every coclosed element in $A$ is Galois.
\end{cor}

\begin{proof} Note that (i) is equivalent to the condition: $\overline{\beta(b)}=\beta(b)$ for every
coclosed element $b\in B$. 

First assume (i). Let $a\in A$ be a coclosed element. By Theorems \ref{t:main} and \ref{t:main2} (ii) we have
$a=\overline{\beta\alpha(a)}$. By hypothesis we have $\overline{\beta\alpha(a)}=\beta\alpha(a)$. Hence 
$a=\beta\alpha(a)$, and so $a$ is Galois.

Conversely, assume (ii). Let $b\in B$ be a coclosed element. By hypothesis the coclosed element $\overline{\beta(b)}\in
A$ is Galois, hence $\overline{\beta(b)}=\beta\alpha(\overline{\beta(b)})$. By Theorems \ref{t:main}
and \ref{t:main2} (ii) we have $\alpha(\overline{\beta(b)})=b$. Hence $\overline{\beta(b)}=\beta(b)$.
\end{proof}

\begin{rem} \rm (1) Under the hypotheses of Theorem \ref{t:main2}, every coclosed element in $B$ is Galois by Lemma
\ref{l:cocgal}. In case every element in $A$ is Galois, Theorem \ref{t:main2} establishes in fact a bijection between
the sets of coclosed Galois elements in $A$ and $B$. We point out that this is not true in general. Indeed,
let $(\alpha,\beta)$ be the UCC cosmall Galois connection from Example \ref{e:ex} (1). Then $\alpha(G)=G$
and $\beta$ preserves finite suprema. The coclosed Galois elements in the domain $A=S(G)$ of $\alpha$ are $H_3$ and $G$,
whereas the coclosed Galois elements in the codomain $B=S(G)$ of $\alpha$ are $0$, $H_2$, $H_3$ and $G$. Hence there is
no bijection between the two sets of coclosed Galois elements. Note that there are elements in $A=S(G)$ which are not
Galois, for instance $H_2$. 

(2) Consider again the UCC cosmall Galois connection from Example \ref{e:ex} (1). Then $\alpha(G)=G$ and $\beta$
preserves finite suprema. Also, $A=S(G)$ is amply supplemented modular, as any subgroup lattice of a finite group. But
condition (i) in Corollary \ref{c:bij} does not hold, because for instance $H_3$ is coclosed in $S(G)$, but
$\beta(H_3)=H_4$ is not coclosed in $S(G)$. Obviously, neither condition (ii) in Corollary \ref{c:bij} holds, because
for instance $H_2$ is a coclosed element in $A=S(G)$ which is not Galois. 
\end{rem}

\section{Applications}

In this section we apply the above results to submodule lattices of suitable modules, which are clearly bounded modular
lattices. Let us first identify some Galois connection between submodule lattices. 

For modules $M$ and $X$, we denote by $\Gen(M)$ the set of $M$-generated modules, and by ${\rm
Tr}_M(X)=M\Hom_R(M,X)$ the greatest submodule of $X$ which is $M$-generated. 

Let $M$ and $N$ be left $R$-modules, and let $S=\End_R(M)$. Consider the following functions between the complete
modular lattices $\mathcal{S}_R(N)$ and $\mathcal{S}_S(\Hom_R(M,N))$ of left $R$-submodules of $N$ and left
$S$-submodules of $\Hom_R(M,N)$ respectively:
\begin{align*}  &\alpha:\mathcal{S}_S(\Hom_R(M,N))\to \mathcal{S}_R(N), \quad \alpha(Y)=MY, \quad \\
&\beta:\mathcal{S}_R(N)\to \mathcal{S}_S(\Hom_R(M,N)),\quad \beta(X)=\Hom_R(M,X). \end{align*} Then $(\alpha,\beta)$ is
a Galois connection \cite[Proposition~3.4]{AN}. The coclosed elements in $\mathcal{S}_R(N)$ are exactly the coclosed
submodules of $N$. A submodule $X$ of $N$ is a Galois element in $\mathcal{S}_R(N)$ if and only if $X\in \Gen(M)$
\cite[Proposition~3.5]{AN}. Now assume that $M$ is finitely generated quasi-projective and $N\in \Gen(M)$. Then
$\alpha(\Hom_R(M,N))=N$, $\beta$ preserves finite suprema and $Y=\Hom_R(M,MY)=\beta\alpha(Y)$ for every submodule $Y$ of
$\Hom_R(M,N)$ by \cite[Proposition~4.9]{AN} or \cite[18.4]{Wis}. Hence every submodule of $\Hom_R(M,N)$ is Galois. 

Let $\sigma[M]$ be the full subcategory of the category of left $R$-modules consisting of all submodules of
$M$-generated modules. Recall that $N\in \sigma[M]$ is called \emph{$M$-faithful} if for every $0\neq g\in \Hom_R(X,N)$
with $X\in \sigma[M]$, $\Hom_R(M,X)g\neq 0$. Note that if $N$ is $M$-faithful, then the above function $\beta$ is
clearly injective. When applied to the bounded modular submodule lattice of a module, Definition \ref{d:lifting} yields
the notions of supplemented and lifting modules. Now Corollary \ref{c:dualg} and Theorem
\ref{t:lifting} give the following corollary, whose first two properties are given in \cite[Theorem~4.2]{GG} and
\cite[43.7]{Wis} respectively. 

\begin{cor} Let $M$ be a finitely generated quasi-projective left $R$-module with $S=\End_R(M)$ and $N\in \Gen(M)$.
Then:

(i) The left $R$-module $N$ and the left $S$-module $\Hom_R(M,N)$ have the same dual Goldie dimension.
 
(ii) $N$ is supplemented if and only if $\Hom_R(M,N)$ is supplemented. 

(iii) If $N$ is lifting, then $\Hom_R(M,N)$ is lifting. If $N$ is $M$-faithful and $\Hom_R(M,N)$ is
lifting, then $N$ is lifting.
\end{cor}

Finally, Theorem \ref{t:main2} yields the following consequence. 

\begin{cor} \label{c:main} Let $M$ be a finitely generated quasi-projective left $R$-module with $S=\End_R(M)$, 
$N\in \Gen(M)$, and let $\alpha$, $\beta$ be the above functions. Then $\alpha$ and $\beta$ restrict to
order-preserving mutually inverse bijections between the sets of coclosed left $S$-submodules of $\Hom_R(M,N)$ and
coclosed left $R$-submodules of $N$.
\end{cor}

\begin{rem} \rm Our theory of cosmall Galois connections is clearly dualizable to one of some naturally defined
essential Galois connections. As a consequence, one may obtain a lattice-theoretic version of Zelmanowitz's results from
\cite{Z}.
\end{rem}

\end{document}